\documentclass{amsart}
\usepackage{graphicx}

\usepackage{amsmath}
\usepackage{amssymb}
\usepackage{url}
\usepackage{csquotes}
\usepackage[english]{babel}
\usepackage[charterbt]{mathdesign}

\renewcommand{\phi}{\varphi}
\renewcommand{\epsilon}{\varepsilon}

\newcommand{\F}{\mathbf{F}}
\newcommand{\N}{\mathbf{N}}
\newcommand{\R}{\mathbf{R}}
\newcommand{\Z}{\mathbf{Z}}
\newcommand{\e}{\mathrm{e}} 

\newcommand{\s}[2]{\genfrac{\{}{\}}{0pt}{}{#1}{#2}} 

\newcommand{\GH}{{L^2(\R^d, \dx\gamma)}}

\newcommand{\D}{\,\mathrm{d}}
\newcommand{\dx}{\mathrm{d}}

\newtheorem{theorem}{Theorem}
\newtheorem{lemma}{Lemma}
\newtheorem{definition}{Definition}

\begin{document}

\title{On the integral kernels
of derivatives of the Ornstein-Uhlenbeck semigroup}

\author{Jonas Teuwen}

\address{Delft Institute of Applied Mathematics,
   Delft University of Technology\\ P.O. Box 5031\\ 2600 GA Delft\\ The
   Netherlands\\
j.j.b.teuwen@tudelft.nl}

\maketitle

\begin{abstract}
This paper presents a closed-form expression for the integral kernels associated with the
 derivatives of the
 Ornstein-Uhlenbeck semigroup $\e^{tL}$. Our approach is to expand the Mehler
 kernel into Hermite polynomials and applying the powers $L^N$ of the
 Ornstein-Uhlenbeck operator to it, where we exploit the fact
 that the Hermite polynomials are eigenfunctions for $L$. As an application we give an
 alternative proof of the kernel estimates by \cite{Portal2014}, making all
 relevant quantities explicit.
\end{abstract}


\maketitle
\section{Introduction}
Much effort \cite{Harboure2000,Kemppainen2015,Maas2010b,MaasNeervenPortal2011,MauceriMeda2007,MauceriMeda2012,Muckenhaupt,Pineda2008,Portal2014,Sjogren1997,Teuwen2015} has gone into developing the harmonic analysis of the {\em Ornstein-Uhlenbeck operator}
\begin{equation}
 \label{eq:OU-operator}
 L := \frac12 \Delta - \langle x, \nabla \rangle.
\end{equation}
On the space $L^2(\R^d, \dx\gamma)$,
where $\gamma$ is the Gaussian measure
\begin{equation}
 \label{eq:Gaussian-measure}
 \dx\gamma(x) := \pi^{-d/2} \e^{-|x|^2} \dx{x},
\end{equation}
this operator can be viewed as the Gaussian counterpart of the Laplace operator $\Delta$.
Indeed, one has $L = -\nabla^*\nabla$, where $\nabla$ is the the usual gradient
and the $\nabla^*$ is its adjoint in $L^2(\R^d, \dx\gamma)$.
It is a classical fact that the semigroup operators $e^{tL}$, $t>0$, are integral operators of the form
\begin{equation*}
  \e^{tL} u(\cdot) = \int_{\R^d} M_t(\cdot, y) u(y) \D\gamma(y),
\end{equation*}
where $M_t$ is the so-called {\em Mehler kernel} \cite{Sjogren1997}
\begin{equation}
  \label{eq:Mehler-kernel-oneformula}
  \begin{aligned}
  M_t(x, y)
  & = \frac{\exp\Bigl(\displaystyle-\frac{|\e^{-t}x - y|^2}{1 - \e^{-2t}}
    \Bigr)}{(1 - \e^{-2t})^{d/2}} \e^{|y|^2}
  \end{aligned}
\end{equation}
(see \cite{Teuwen2015} for a representation of $M_t$ which makes the symmetry in $x$ and $y$
explicit).

Developing a Hardy space theory for $L$ is the subject of active
current research \cite{Portal2014,MauceriMeda2007}.
In this theory the derivatives $(d^k/dt^k)e^{t L} = L^N e^{tL}$ play an important role.
The aim of the present paper is to derive closed form expressions for the integral
kernels of these derivatives, that is, to determine explicitly the kernels $M_t^N$ such that we have the identity
\begin{equation}\label{eq:Mehler}
  L^N\e^{tL} u(\cdot)  = \int_{\R^d} M_t^N(\cdot, y) u(y) \D\gamma(y).
\end{equation}
Direct application of the derivatives $d^N/dt^N$ to the Mehler kernel yields expressions which become
intractible even for small values of $N$.
Our approach will be to expand the Mehler kernel in terms of the $L^2$-normalised Hermite polynomials and then
to apply $L^N$ to it, thus exploiting the fact that the Hermite polynomials are eigenfunctions for $L$.

As an application of our main result, which is proved in section \ref{sec:mainresult}
after developing some preliminary material in the
sections \ref{sec:prelim}-\ref{sec:Weyl},
we shall give a direct proof for the kernel bounds of \cite{Portal2014} in section \ref{sec:application}.

\section{Preliminaries}\label{sec:prelim}

In this preliminary section we collect some standard properties of
Hermite polynomials and their relationship with the Ornstein-Uhlenbeck operator.
Most of this material is classical and can be found in \cite{Sjogren1997,MR1215939}.

\subsection{Hermite polynomials}
The {\em Hermite polynomials} $H_n$, $n\ge 0$, are
defined by Rodrigues's formula
\begin{equation}
  \label{eq:Hermite-Rodrigues}
  H_n(x) := (-1)^n \e^{x^2} \partial_x^n \e^{-x^2}.
\end{equation}
Their $L^2$-normalizations,
\[ h_n := \frac{H_n}{\sqrt{2^{n} n!}}\]
 form an orthonormal basis for $L^2(\R,\dx\gamma)$.
We shall use the fact that the generating function for the Hermite polynomials is given by
\begin{equation}
 \label{eq:Generating-function-identity} \sum_{n = 0}^\infty \frac{H_n(x)}{n!} t^n =\e^{2 tx - t^2}.
\end{equation}
The relationship with the Ornstein-Uhlenbeck operator is encoded in the eigenvalue identity
$L H_n = -n H_n$, from which it follows that for all $t \ge 0$ we have $\e^{tL} H_n =
\e^{-t n} H_n$. From this
one quickly deduces that the Mehler kernel is given by
\begin{equation}
  \label{eq:Mehler_kernel_non_comp}
  M_t(x, y) := \sum_{n = 0}^\infty \e^{-t n} h_n(x)h_n(y).
\end{equation}
We will need two further identities for the Hermite polynomials
which can be found, e.g., in \cite[Chapter 18]{NIST:DLMF}: the integral representation
\begin{equation}
  \label{eq:Hermite-integral}
  H_n(x) = \frac{(-2i)^n \e^{x^2}}{\sqrt\pi} \int_{-\infty}^\infty
  \xi^n \e^{2ix \xi} \e^{-\xi^2} \D\xi
\end{equation}
and the ``binomial'' identity
\begin{equation}
  \label{eq:Hermite-binomial-type}
  H_n(x + y) = \sum_{k = 0}^n \binom{n}{k} (2y)^{n - k} H_k(x) .
\end{equation}

\subsection{Hermite polynomials in several variables}
The Hermite polynomials on $\R^d$ are defined, for
 multiindices
$\alpha = (\alpha_1, \dots, \alpha_d)\in\N^d$ by the tensor extensions
\begin{equation}
  \label{eq:Hermite-Rodrigues-Rd}
  H_\alpha(x) := \prod_{n = 1}^{d} h_{\alpha_i}(x_i),
\end{equation}
for $x=(x_1,\dots,x_d)$ in $\R^d$.
The normalized Hermite polynomials
\begin{equation}
  \label{eq:Hermite-normalized}
  h_\alpha := \frac{H_\alpha}{\sqrt{2^{|\alpha|} \alpha!}}, \  \text{  where } \alpha! = \alpha_1! \cdot \dots \cdot \alpha_d!
\end{equation}
form an
orthonormal basis in $\GH$, and we have the eigenvalue identity
\begin{equation}
  \label{eq:OU-Hermite-action}
  LH_{\alpha} = -|\alpha| H_\alpha, \ \text{  where } |\alpha| = \alpha_1 + \dots + \alpha_d.
\end{equation}
If we consider the action of $L^N \e^{tL}$ on a Hermite
polynomial $h_\alpha$,  through the multinomial theorem applied to $|\alpha|^k$ we get
(writing $L_d$ for the operator $L$ in dimension $d$ and $L_1$ for the operator $L$ in dimension $1$)
\begin{equation}
  \label{eq:reduction-to-d1}
\begin{aligned}  L_d^N \e^{tL} & h_\alpha(x) = |\alpha|^N \e^{-t|\alpha|}
                            h_{\alpha_1}(x_1) \cdot \hdots \cdot
                            h_{\alpha_d}(x_d)\\
&= \sum_{|n| = N} \binom{N}{n_1, n_2, \hdots, n_d} \alpha_1^{n_1}
    \cdot \hdots \cdot \alpha_d^{n_d} \e^{-t\alpha_1}\cdot\hdots\cdot \e^{\alpha_d}
                            h_{\alpha_1}(x_1) \cdot \hdots\cdot
                            h_{\alpha_d}(x_d)
\\ & = \sum_{|n| = N} \binom{N}{n_1, n_2, \hdots, n_d} L_1^{n_1}
      \e^{tL_1} h_{\alpha_1}(x_1) \cdot \hdots \cdot L_1^{n_d}
      \e^{tL_1} h_{\alpha_d}(x_d).
\end{aligned}
\end{equation}
This implies that we can reduce the question of computing the
$d$-dimensional version of the integral kernel to the one-dimension one.

\section{A combinatorial lemma}\label{sec:setup}
From now on we concentrate on the Ornstein-Uhlenbeck operator $L$ in one dimension, i.e., in $L^2(\R,\dx\gamma)$.
We are going to follow the approach of \cite{Sjogren1997}.
Recalling the identity $L h_n = -n h_n$, we will
apply $L^N$ to the generating function of the Hermite polynomials
\eqref{eq:Generating-function-identity}. A problem which
immediately occurs is that $\Delta$ and $\langle x, \nabla \rangle$ do not
commute, and because of this we cannot use a standard binomial formula
to evaluate $L^N$. Instead, we note that
\[L g = -t \partial_t g.\] In particular this implies that
\[L^N g = (-1)^k D_N g,\] where
\begin{equation}
  \label{eq:Differential-operator-generated}
  D_N := \underbrace{t \partial_t \circ t \partial_t \circ \dots \circ t \partial_t}_{\text{$k$ times}} = (t\partial_t)^k.
\end{equation}
The following lemma will be very useful.
\begin{lemma}\label{lem:Lk-powers-expanded}
 We have
  \begin{equation}
    \label{eq:Expanded-Differentiatial-operator-generated}
    D_N = \sum_{n = 0}^N \s{N}{n} t^n \partial_t^n, 
  \end{equation}
  where $\s{N}{n}$ are the Stirling numbers of the second kind.
\end{lemma}
The Stirling numbers of the second kind are quite well-known combinatorial objects.
For the sake of completeness we will state their definition and
recall some relevant properties below. For more information we refer the reader to \cite{LINT}.
The related Stirling numbers of the first kind will not be needed here.

We begin by recalling the definition of {\em falling factorial}
\begin{equation}
  \label{eq:falling-factorial}
  (j)_n := j(j - 1)\dots (j - n + 1) = \frac{j!}{(j - n)!},
\end{equation}
for non-negative integers $k \geq n$.
\begin{definition}\label{def:stirling-numbers-second-kind}
For non-negative integers $N \ge n$, the number {\em Stirling number of the second kind} $\s{N}{n}$
is defined as the number of partitions of an $N$-set into
  $n$ non-empty subsets.
  \end{definition}
These numbers satisfy the recursion identity
  \begin{equation}
    \label{eq:Stirling-numbers-second-kind-recursion}
    \s{N}{n} = N \s{N - 1}{n} + \s{N - 1}{n - 1}.
  \end{equation}
  For all non-negative integers $j$ and $k$ one has the generating function identity
  \begin{equation}
    \label{eq:Stirling-numbers-second-kind-generating-function-1}
    j^N = \sum_{n = 0}^N \s{N}{n} (j)_n.
  \end{equation}

\section{Weyl Polynomials}\label{sec:Weyl}
Before turning to the proof of lemma
\ref{lem:Lk-powers-expanded}, let us already mention that it only depends
on the commutator identity $[t,\partial_t] = -1$.
This brings us to the
observation that \emph{Weyl polynomials} provide the natural
habitat for our expressions.
Rougly speaking, a Weyl polynomial is
a polynomial in two non-commuting variables $x$ and $y$
which satisfy the commutator identity $[x, y] = -1$.
This is made more precise in the following definition.

\begin{definition}
The {\em Weyl algebra} over a field $\F$ of characteristic zero is the
ring $\F\langle x,y\rangle$ of all polynomials of the form
$p(x,y) = \sum_{m=0}^M \sum_{n=0}^N c_{mn} x^m y^n$ with coefficients $c_{mn}\in\F$
in two noncommuting variables $x$ and $y$ which
satisfy the commutator identity \[[x,y]:=xy-yx = -1.\]
\end{definition}

We now have the following abstract version of lemma \ref{lem:Lk-powers-expanded}:
\begin{lemma}\label{lem:Weyl-expansion}
  In the Weyl algebra $\F \langle x, y\rangle$ we have the identity
  \begin{equation}
    \label{eq:Weyl-base-thm}
    (x y)^m = \sum_{i = 1}^m \s{m}{i} x^i y^i,
  \end{equation}
where $\s{m}{i}$ are the Stirling numbers of the second kind.
\end{lemma}
As a preparation for the proof of lemma \ref{lem:Weyl-expansion} we make a couple
of easy computations. If we set $D := xy$, then
\begin{align}
  D x^m &= x^m D + m x^m\label{eq:Weyl-push-1},\\
  D y^m &= y^m D - m y^m\label{eq:Weyl-push-2}.
\end{align}
This can be shown by induction on $m$. For instance, note that
\begin{equation*}
  D x^m = x(D + 1)x^{m - 1} = x D x^{m - 1} + x^m.
\end{equation*}
If we take this a bit further and have $p \in \F[D]$, then
\begin{align}
  \label{eq:Weyl-poly-1}
  p(D) x^m &= x^m p(D + m),\\
  \label{eq:Weyl-poly-2}
  p(D) y^m &= y^m p(D - m).
\end{align}
The $m$-th powers, $m\ge 1$, of $x$ and $y$ satisfy
\begin{align}
  \label{eq:Weyl-xmym}
  x^m y^m &= \prod_{i = 0}^{m - 1} (D - i),\\
  \label{eq:Weyl-ymxm}
  y^m x^m &= \prod_{i = 1}^m (D + i).
\end{align}
This can be seen using induction:
\begin{equation*}
  x^{m + 1} y^{m + 1} = x^m D y^m \overset{\eqref{eq:Weyl-poly-2}}{=} x^m y^m (D - m)
\end{equation*}
and
\begin{equation*}
  y^{m + 1} x^{m + 1} = y^m (D + 1) x^m = y^m D x^m + y^m x^m \overset{\eqref{eq:Weyl-poly-1}}{=} y^m x^m (D + (m + 1)).
\end{equation*}
\begin{definition}
  The {\em weighted degree} of a monomial $x^m y^n\in \F\langle x,y\rangle$ is the integer $m-n$.
  A polynomial in $\F\langle x,y \rangle$ is said to be {\em homogeneous of weighted degree $j$}
  if all its constituting monomials have weighted degree $j$.
\end{definition}
Left multiplication by $xy$ is {\em homogeneity preserving}, i.e., for all $j\in\Z$
it maps the set of homogeneous monomials of weighted degree $j$ into itself. To prove this,
first consider a monomial $x^m y^n$ of dweighted egree $j = m-n$. Then,
\begin{equation*}
  (xy) x^m y^n \mathrel{\overset{\eqref{eq:Weyl-push-1}}{=}} (x^m (xy) + m x^m)y^n = x^m(xy)y^n +mx^m y^n = x^{m + 1} y^{n + 1} + m x^m y^n,
\end{equation*}
and we see that weighted degree of homogeneity is indeed preserved. The general case follows immediately.
Through \eqref{eq:Weyl-xmym} we conclude that left multiplication
$x^k y^k$  is homogeneity preserving as well, for all non-negative
integers $k$. We claim that left multiplication by $x^i y^j$ is homogeneity preserving
only if $i = j$. To see this note that \[y x^i y^j = x^i y^{i + j} + i
x^{i - 1} y^j\] from which we can deduce that
\begin{equation*}
  x^m y^M x^n y^N = x^{n + m} y^{N + M} + \text{ lower order terms}.
\end{equation*}
From which the claim follows.

Finally, a polynomial is homogeneity preserving if and only if all of
its constituting monomials have this property. If this were not to be the
case we could look at the highest-order non-homogeneous term and note
from above $x^m y^M x^n y^N$ would give terms of a lower order in the polynomial expansion which
cannot cancel as they have different powers of $x$ or $y$.

It follows from these observations that
\begin{equation}
  \label{eq:Weyl-G0}
  F_0 := \biggl\{\sum_{n=0}^N c_{n} x^n y^n \ \bigg| \ N\in\N, \, c_1,\dots,c_N\in\F \biggr\}
\end{equation}
is precisely the set of \emph{homogeneity preserving polynomials} in $\F\langle
x, y\rangle$.

Now everything is in place to give the proof of lemma
\ref{lem:Weyl-expansion}.

\begin{proof}[Proof of lemma
\ref{lem:Weyl-expansion}] As $(xy)^k$ is homogeneity preserving,
we infer that there are coefficients $a_i^k$ in $\F$ such
that
\begin{equation}
  \label{eq:Weyl-lemma-intermediate-step-wanted-form}
  (xy)^k = \sum_{i = 0}^k a_i^k x^i y^i.
\end{equation}
We will apply $x^j$ to the right on both sides of
\eqref{eq:Weyl-lemma-intermediate-step-wanted-form} and derive an
expression for the $a_i^k$. First note that \eqref{eq:Weyl-poly-1}
gives
\begin{equation*}
  (xy)^k x^j = x^j (xy + j)^k,
\end{equation*}
and \eqref{eq:Weyl-xmym} together with \eqref{eq:Weyl-poly-1} gives
\begin{equation*}
  x^i y^i x^j \overset{\eqref{eq:Weyl-xmym}}{=}
\prod_{\ell = 0}^{i - 1} (xy - \ell) x^j \overset{\eqref{eq:Weyl-poly-1}}{=} x^j \prod_{\ell = 0}^{i - 1} (xy - \ell + j).
\end{equation*}
Hence, to find the coefficients $a_i^k$ it is sufficient to consider
\begin{equation*}
  (xy + j)^k = \sum_{i = 0}^k a_n^k \prod_{\ell = 0}^{i - 1}(xy - \ell + j).
\end{equation*}
Comparing the constant terms on both the left-hand side and right-hand side,
we find
\begin{equation}
  \label{eq:Weyl-result-generating-function}
  j^k = \sum_{i = 0}^k a_i^k \prod_{\ell = 0}^{i - 1} (j - \ell) = \sum_{i = 0}^k a_i^k (j)_i,
\end{equation}
where $(j)_i$ is the falling factorial as in
\eqref{eq:falling-factorial}. Comparing
\eqref{eq:Weyl-result-generating-function} with the generating function
of the Stirling numbers of the second kind $\s{k}{i}$ as given in
\eqref{eq:Stirling-numbers-second-kind-generating-function-1}, we see that
 $a_i^k = \s{k}{i}$. This concludes the proof of
lemma \ref{lem:Weyl-expansion}.
\end{proof}

\section{The integral kernel of $L^N \e^{t L}$}\label{sec:mainresult}
As mentioned before, as a first step we would like to apply $D_N$ to the generating
function $g(x,t) := \e^{-2tx + t^2 } = \e^{-(x - t)^2 + x^2}$ for the Hermite polynomials
\eqref{eq:Generating-function-identity}. We first compute
the action of $\partial_t^N$ on the generating function.
\begin{lemma}
  We have
  \begin{equation}
    \label{eq:derivatives-generating-function-Hermite}
    \partial_t^N \e^{-(x - t)^2 + x^2}  = \e^{-(x - t)^2 + x^2} H_N(x - t).
  \end{equation}
\end{lemma}
\begin{proof}
 We first note  that,
\begin{align*}
  \partial_t \e^{-(x - t)^2} &= 2(x - t) \e^{-(x - t)^2} = - \partial_x \e^{-(x - t)^2}.
\end{align*}
Using this we get
\begin{align*}
  \partial_t^N \e^{-(x - t)^2 + x^2} &= \e^{x^2} \partial_t^N \e^{-(x -
  t)^2}\\
  &= \e^{x^2} \partial_t^{N - 1} \partial_t \e^{-(x - t)^2}\\
  &= -\e^{x^2} \partial_t^{N - 1} \partial_x \e^{-(x - t)^2}\\
  &= (-)^2\e^{x^2} \partial_t^{N - 2} \partial_x^2 \e^{-(x - t)^2}\\
  &= \dots\\
  &= (-1)^N \e^{x^2} \partial_x^N \e^{-(x - t)^2}.
\end{align*}
By a change of variables,
\begin{equation*}
  \partial_t^N \e^{-(x - t)^2 + x^2} = (-1)^N \e^{-(x - t)^2 + x^2} \e^{(x - t)^2} \partial_y^N \e^{-y^2} \Bigr|_{y = x - t}.
\end{equation*}
Hence, by \eqref{eq:Hermite-Rodrigues},
\begin{equation*}
  \partial_t^N \e^{-(x - t)^2 + x^2} = \e^{-(x - t)^2 + x^2} H_N(x - t).
\end{equation*}
\end{proof}
\begin{lemma}
  \label{lem:power-Ornstein-Uhlenbeck-generating-Hermite}
  For all $x\in\R$ and $t>0$ we have
  \begin{equation}
    \label{eq:power-Ornstein-Uhlenbeck-generating-Hermite}
    L^N \e^{-(x - t)^2 + x^2} = (-1)^N \e^{-(x - t)^2 + x^2} \sum_{n = 0}^N \s{N}{n} t^n H_n(x - t).
  \end{equation}
\end{lemma}
\begin{proof}
This is now easy to prove. Recalling that $L = -t \partial_t$ and using \eqref{eq:Expanded-Differentiatial-operator-generated},
we get
\begin{align*}
  L^N \e^{-(x - t)^2 + x^2} &= D_N \e^{-(x - t)^2 + x^2}\\
  &= (-1)^N \sum_{n = 0}^N \s{N}{n} t^n \partial_t^n \e^{-(x - t)^2 + x^2}\\
  &= (-1)^N \e^{-(x - t)^2 + x^2} \sum_{n = 0}^N \s{N}{n} t^n H_n(x - t).
\end{align*}
\end{proof}
Our next theorem is the main result of this paper and provides an explicit expression for the
integral kernel of $L^N e^{t L}$.
\begin{theorem}\label{th:integral-kernel}
  Let $L$ be the Ornstein-Uhlenbeck operator on $L^2(\R^d,\dx \gamma)$,
let $t>0$, and let $N\ge 0$ be an integer. The integral
  kernel $M_t^N$ of $L^N \e^{tL}$  is given by
  \begin{equation}
    \label{eq:Mehler-kernel-of-powers}
  \begin{split}
  M_t^N(x, y) &= M_t(x, y) \sum_{|n| = N} \binom{N}{n_1, \dots, n_d} \prod_{i = 1}^d
  \sum_{m_i = 0}^{n_i} \sum_{\ell_i = 0}^{m_i}  2^{-m_i}
  \s{n_i}{m_i} \binom{m_i}{\ell_i}\\
  &\quad \times \biggl(-\frac{\e^{-t}}{\sqrt{1 - \e^{-2t}}} \biggr)^{2m_i - \ell_i}
H_{\ell_i}(x_i) H_{2m_i - \ell_i}\biggl(\frac{x_i \e^{-t} - y_i}{\sqrt{1 - \e^{-2t}}}\biggr).
  \end{split}
  \end{equation}
\end{theorem}
\begin{proof}
We first prove the result for $d = 1$.
Pulling $L^N$ through the integral expression \eqref{eq:Mehler} for $e^{t L}$ involving the Mehler kernel, we must
find a suitable expression for the kernel $M_t^N(\cdot, y) = L^N
M_t(\cdot, y)$.  Using \eqref{eq:Mehler_kernel_non_comp} and the normalization of $H_m$ in
\eqref{eq:Hermite-normalized} we get
\begin{align*}
  M_t^N(x, y) &= L^N \sum_{m = 0}^\infty \frac{\e^{-t m}}{m!}\frac1{2^m} H_m(x) H_m(y)\\
  &\overset{\eqref{eq:Hermite-integral}}{=} L^N \sum_{m = 0}^\infty \frac{\e^{-t m}}{m!}\frac1{2^m} H_m(x) \frac{(-2i)^m}{\sqrt{\pi}} \e^{y^2} \int_{-\infty}^\infty \e^{-\xi^2} \xi^m \e^{2 i y \xi} \D\xi\\
  &= L^N \frac{\e^{y^2}}{\sqrt{\pi}} \int_{-\infty}^\infty  \e^{-\xi^2} \e^{2 i y \xi} \sum_{m = 0}^\infty \frac1{m!} H_m(x) (-i \xi \e^{-t})^m\D\xi\\
  &\overset{\eqref{eq:Generating-function-identity}}{=} L^N
  \frac{\e^{y^2}}{\sqrt{\pi}} \int_{-\infty}^\infty \e^{-\xi^2} \e^{2
    i y \xi} \e^{-(x + i \xi \e^{-t})^2 + x^2} \D\xi.
\end{align*}
The operator $L^N$ is applied with respect to $x$ here,
so by lemma \ref{lem:power-Ornstein-Uhlenbeck-generating-Hermite} we get
\begin{align*}
  M_t^N(x, y) &= \frac{\e^{y^2}}{\sqrt{\pi}}
  \int_{-\infty}^\infty \e^{-\xi^2} \e^{2 i y \xi} L^N \e^{-(x + i \xi \e^{-t})^2 + x^2} \D\xi\\
  &\overset{\eqref{eq:power-Ornstein-Uhlenbeck-generating-Hermite}}{=}(-1)^N
  \frac{\e^{x^2 + y^2}}{\sqrt{\pi}} \sum_{m = 0}^N \s{N}{m}
  \int_{-\infty}^\infty \e^{2 i y \xi} \e^{-(x + i \xi \e^{-t})^2} (i
  \xi \e^{-t})^m H_m(y + i \xi \e^{-t}) \e^{-\xi^2} \D\xi,
\end{align*}
where in last line we have used the analytic continuation of the algebraic
identity \eqref{eq:power-Ornstein-Uhlenbeck-generating-Hermite}. Similarly we can expand $H_m(y + i \xi \e^{-t})$ using
\eqref{eq:Hermite-binomial-type}. This gives
\begin{equation*}
  H_m(y + i \xi \e^{-t}) = \sum_{\ell = 0}^m \binom{m}{\ell} H_\ell(y)
  (2 i \xi \e^{-t})^{m - \ell},
\end{equation*}
so that $M_t^N$ can be written as
\begin{equation}
  \label{eq:M_k-integral}
(-1)^N  \frac{\e^{x^2 + y^2}}{\sqrt{\pi}} \sum_{m = 0}^N \sum_{\ell = 0}^m \s{N}{m} \binom{m}{\ell} H_\ell(y) 2^{m - \ell} \int_{-\infty}^\infty \e^{2 i y \xi - \xi^2} \e^{-(x + i \xi \e^{-t})^2}  (i \xi \e^{-t})^{2m - \ell} \D\xi.
\end{equation}
If we set $M = 2m - \ell$, this reduces our task to computing the
integral
\begin{align}
  \notag
  \frac{\e^{x^2 + y^2}}{\sqrt{\pi}} \int_{-\infty}^\infty &\e^{2 i y \xi - \xi^2} \e^{-(x + i \xi \e^{-t})^2}  (i \xi \e^{-t})^M \D{\xi}\\
  \label{eq:Hermite-integral-derivation-1}
  &= \frac{\e^{y^2}}{\sqrt{\pi}} \int_{-\infty}^\infty \e^{2 i (x \e^{-t} - y) \xi} \e^{-(1 - \e^{-2 t}) \xi^2} (i \xi \e^{-t})^M
  \D{\xi}.
\end{align}
To make the computation less convolved, let us set
\begin{equation*}
  \alpha_t := \sqrt{1 - \e^{-2t}}, \text{ and, } \beta_t(x, y) := \frac{x \e^{-t} - y}{\sqrt{1 - \e^{-2t}}}.
\end{equation*}
This allows us to write the exponential in the integral \eqref{eq:Hermite-integral-derivation-1} as
\begin{equation*}
  \e^{2 i (x\e^{-t} - y) \xi} \e^{-(1 - \e^{-2 t}) \xi^2} = \e^{2 i \alpha_t \beta_t(x,y) \xi} \e^{-\alpha_t^2 \xi^2}.
\end{equation*}
This reduces the problem, after the substitution $\alpha_t \xi \to \xi$,
to computing the integral
\begin{equation*}
  \frac{\e^{y^2}}{\sqrt{\pi}} \frac{i^M \e^{-M t}}{\alpha_t^{M + 1}}
  \int_{-\infty}^\infty \e^{2 i \beta_t(x, y) \xi} \e^{-\xi^2} \xi^M \D{\xi}.
\end{equation*}
The final integral is an integral expression for the
Hermite polynomials \eqref{eq:Hermite-integral}, so
\begin{align*}
\ &   \frac{\e^{y^2}}{\sqrt{\pi}} \frac{i^M \e^{-M t}}{\alpha_t^{M + 1}}
  \int_{-\infty}^\infty \e^{2 i \beta_t(x,y) \xi} \e^{-\xi^2} \xi^M \D{\xi}
\\ & \qquad  \overset{\eqref{eq:Hermite-integral}}{=} \e^{y^2-\beta_t(x,y)^2}
  \frac1{\alpha_t^{M + 1}} \frac{(-1)^M \e^{-M t}}{2^M} H_M(\beta_t(x,y)).
\end{align*}
Next we note that $\exp(y^2 - \beta_t(x,y)^2) \alpha_t^{-1} = M_t$, the
Mehler kernel from \eqref{eq:Mehler}. Hence,
\begin{equation}
  \label{eq:Mehler-kernel-Lk-intermediate-1}
  \begin{split}
  M_t^N(x, y) &= M_t(x, y) \sum_{m = 0}^N \sum_{\ell = 0}^m
  \binom{m}{\ell} \s{N}{m}
 \biggl(-\frac{\e^{-t}}{\sqrt{1 - \e^{-2t}}} \biggr)^{2m - \ell}
  2^{-m}\\
  &\quad \times H_\ell(x) H_{2m - \ell}\biggl(\frac{x \e^{-t} - y}{\sqrt{1 - \e^{-2t}}}\biggr).
  \end{split}
\end{equation}
Applying \eqref{eq:reduction-to-d1} we get the result in $d$ dimensions.
\end{proof}

\section{An application}\label{sec:application}
As an application of our main result, in this section we give an alternative proof of the
bounds on the kernels $K$ and $\tilde K$ of
\cite{Portal2014} (see the definition below), making the dependence on the parameters more explicit.
These kernels play a central role in the study of the Hardy space $H^1(\R^d,\dx \gamma)$ in \cite{Portal2014},
where the standard Calder\'on reproducing formula is replaced by
\begin{equation}
 u = C \int_0^\infty (t^2 L)^{N + 1} \e^{\frac{t^2}{\alpha} L} u \frac{\D{t}}{t} + \int_{\R^d} u \D\gamma,
\end{equation}
where $C$ is a suitable constant only depending on $N$ and
$\alpha$ (this can be seen by letting $u$ be a Hermite polynomial).
The kernels $K$ and $\tilde K$ then occur in several decompositions, and the estimates below allow
them to be related to classical results about the Mehler kernel.

\begin{definition}
  We define the kernels $K$ and $\tilde K$ by
  \begin{align*}
    \int_{\R^d} K_{t^2, N, \alpha}(x, y) u(y) \D\gamma(y) &= (t^2 L)^N
                                                            \e^{\frac{t^2}{\alpha}
                                                            L} u(x),\\
    \int_{\R^d} \tilde{K}_{t^2, N, \alpha, j}(x, y) u(y) \D\gamma(y) &= (t^2 L)^N
                                                            \e^{\frac{t^2}{\alpha}
                                                            L} t \partial_{x_j}^* u(x).
  \end{align*}
\end{definition}
Note that the operators on the right-hand sides are indeed given by integral kernels: the first is a scaled version
of the operator we have already been studying, and a
duality argument implies that the second is given by the integral kernel \[\tilde{K}_{t^2, N, \alpha,j}(x, y) =
t \partial_{x_j} K_{t^2, N, \alpha}(x, y).\]
Thus, both kernels are given as appropriate derivatives of the Mehler kernel.

We begin with a technical lemma which is a rephrased version of
\cite[Lemma 3.4]{Portal2014}. One should take note that we define the kernels with respect
to the Gaussian measure whereas, \cite{Portal2014} defines these with
respect to the Lebesgue measure.
\begin{lemma}
  \label{lem:Mehler-alpha-efficient1}
  For all $\alpha > 1$ and all $t>0$ and $x, y$ in $\R^d$ we have
  \begin{equation}
    \label{eq:Mehler-alpha-efficient2}
   \frac{|\e^{-\frac{t}{\alpha}}x - y|^2}{1 - e^{-2\frac{t}{\alpha}}}
   \geq\frac{\alpha}2 \e^{-2t}  \frac{|\e^{-t}x - y|^2}{1 -
            \e^{-2t}} - \frac{t^2 \min{(|x|^2, |y|^2)}}{1 - \e^{-2\frac{t}{\alpha}}}.
  \end{equation}
  Additionally, we have
    \begin{equation}
    \label{eq:Mehler-alpha-efficient3}
    \alpha \e^{-2t} \leq \frac{1 - \e^{-2t}}{1 - \e^{-2\frac{t}{\alpha}}} \leq \alpha.
  \end{equation}
\end{lemma}
\begin{theorem}
  Let $N$ be a positive integer, $0 < t < T$. The for all large enough $\alpha>1$
  we have
  \begin{enumerate}
  \item If $t |x| \leq C$, then \[\displaystyle |K_{t^2, \alpha, N}(x, y)| \lesssim_{T, N}
  \alpha \exp(\frac{\alpha}2C^2)
  M_{t^2}(x, y)  \exp\biggl(-\frac{\alpha}{8\e^{2T}}  \frac{|\e^{-t^2}x - y|^2}{1 -
            \e^{-2t^2}} \biggr).\]
  \item If $t |x| \leq C$, then
  \[\displaystyle |\tilde{K}_{t^2, \alpha, N,j}(x, y)|
  \lesssim_{T, N} \alpha \exp(\frac{\alpha}2C^2) M_{t^2}(x, y) \exp\biggl(-\frac{\alpha}{8 \e^{2T}}  \frac{|\e^{-t^2}x - y|^2}{1 -
            \e^{-2t^2}} \biggr).\]
  \end{enumerate}
\end{theorem}
\begin{proof}
For $K_{t^2, \alpha, N}$, we use Theorem~\ref{th:integral-kernel} to obtain,
after taking absolute values,
\begin{align*}
  |K_{t^2, \alpha, N}(x, y)| &\leq  M_{\frac{t^2}{\alpha}}(x, y) \sum_{|k| = N} \binom{N}{n_1, \dots, n_d}
  \prod_{i = 1}^d  t^{2k_i }\sum_{\ell_i = 0}^{n_i} \sum_{m_i = 0}^{m_i} 2^{-m_i}
  \binom{m_i}{\ell_i}  \s{n_i}{m_i}\\
  &\quad \times \biggl(\frac{\e^{-\frac{t^2}{\alpha}}}{\sqrt{1 - \e^{-2\frac{t^2}{\alpha}}}} \biggr)^{2m_i - \ell_i}
| H_{\ell_i}(x_i)| \biggl| H_{2m_i - \ell_i}\biggl(\frac{x_i \e^{-\frac{t^2}{\alpha}} - y_i}{\sqrt{1 - \e^{-2\frac{t^2}{\alpha}}}}\biggr) \biggr |.
\end{align*}
Recalling that $\ell_1 + \dots + \ell_d \leq N$, using the assumptions $t \leq T$ and $t |x| \leq
C$ we can bound $t^{2k_i} |H_{\ell_i}(x)|$
by considering the highest order
term to obtain \[t^{2k_i} |H_{\ell_i}(x)| \lesssim_{C,N,T} 1.\]
Using \eqref{eq:Mehler-alpha-efficient3} we proceed by looking at
\begin{align*}
  M_{\frac{t^2}{\alpha}}(x, y) &= M_{t^2}(x, y) \bigg(\frac{{1 - \e^{-2t^2}}}{{1 - \e^{-2\frac{t^2}{\alpha}}}}\bigg)^{1/2}\exp\biggl(\frac{|\e^{-t^2} x - y|^2}{1 - \e^{-2t^2}}
  \biggr) \exp\biggl(-\frac{|\e^{-\frac{t^2}{\alpha}} x - y|^2}{1 - \e^{-2\frac{t^2}{\alpha}}}
  \biggr)\\
&\le \alpha M_{t^2}(x, y) \exp\biggl(\frac{|\e^{-t^2} x - y|^2}{1 - \e^{-2t^2}}
  \biggr) \Bigg[\exp\biggl(-\frac12\frac{|\e^{-\frac{t^2}{\alpha}} x - y|^2}{1 - \e^{-2\frac{t^2}{\alpha}}}
  \biggr) \bigg]^2.
\end{align*}
We can now bound the final Hermite polynomial in the expression of the
kernel. Setting $M_i = 2m_i - \ell_i$ we get
\begin{align*}
 \biggl(\frac{\e^{-\frac{t^2}{\alpha}}}{\sqrt{1 - \e^{-2\frac{t^2}{\alpha}}}} \biggr)^{M_i}
 \biggl|H_{M_i}\biggl(\frac{\e^{-\frac{t^2}{\alpha}}x_i - y_i}{\sqrt{1 -
  \e^{-2\frac{t^2}{\alpha}}}}\biggr)\biggr| &\lesssim_N
                                    \biggl(\frac{\e^{-\frac{t^2}{\alpha}}}{\sqrt{1
                                    - \e^{-2\frac{t^2}{\alpha}}}} \biggr)^{M_i}
                                    \biggl(\frac{|\e^{-\frac{t^2}{\alpha}}x_i - y_i|}{\sqrt{1 -
  \e^{-2\frac{t^2}{\alpha}}}}\biggr)^{M_i}\\
&\le \biggl(\frac{|\e^{-\frac{t^2}{\alpha}}x_i - y_i|}{{1 -
  \e^{-2\frac{t^2}{\alpha}}}}\biggr)^{M_i}.
\end{align*}
Also,
\begin{align*}
 \biggl(\frac{|\e^{-\frac{t^2}{\alpha}}x_i - y_i|}{{1 -
  \e^{-2\frac{t^2}{\alpha}}}}\biggr)^{M_i} \exp\biggl(-\frac1{2}\frac{|\e^{-\frac{t^2}{\alpha}} x_i - y_i|^2}{1 - \e^{-2\frac{t^2}{\alpha}}}
  \biggr) \lesssim 1.
\end{align*}
Putting these estimates together, using Lemma~\ref{lem:Mehler-alpha-efficient1}, and taking
$\alpha>1$  so large that
\[
1-\frac{\alpha}{4 \e^{2T}} \le -\frac{\alpha}{8 \e^{2T}} \ \hbox{ and } \
1 - \e^{-2\frac{t^2}{\alpha}} \ge \frac{t^2}{\alpha},
\] we obtain
\begin{align*}
  |K_{t^2, \alpha, N}(x, y)| & \lesssim_{T, N}\alpha M_{t^2}(x, y) \exp\biggl(\frac{|\e^{-t^2} x - y|^2}{1 - \e^{-2t^2}}
  \biggr) \exp\biggl(-\frac12\frac{|\e^{-\frac{t^2}{\alpha}} x - y|^2}{1 - \e^{-2\frac{t^2}{\alpha}}}
  \biggr)\\
&\leq \alpha M_{t^2}(x, y) \exp\biggl(\frac{|\e^{-t^2} x - y|^2}{1 - \e^{-2t^2}}
  \biggr) \exp\biggl(-\frac{\alpha}{4 \e^{2T}}  \frac{|\e^{-t^2}x - y|^2}{1 -
            \e^{-2t^2}} \biggr) \exp\biggl(\frac12 \frac{t^4 |x|^2}{1 - \e^{-2\frac{t^2}{\alpha}}} \biggr)\\
&= \alpha M_{t^2}(x, y) \exp\biggl(\Big(1-\frac{\alpha}{4 \e^{2T}}\Big)  \frac{|\e^{-t^2}x - y|^2}{1 -
            \e^{-2t^2}} \biggr) \exp\biggl(\frac12 \frac{t^4 |x|^2}{1 - \e^{-2\frac{t^2}{\alpha}}} \biggr)\\
            &\leq \alpha M_{t^2}(x, y)  \exp\biggl(-\frac{\alpha}{8 \e^{2T}}  \frac{|\e^{-t^2}x - y|^2}{1 -
            \e^{-2t^2}} \biggr) \exp\biggl(\frac\alpha2 t^2 |x|^2 \biggr)
            \\  &\leq \alpha \exp\biggl(\frac\alpha2 C^2 \biggr) M_{t^2}(x, y)
            \exp\biggl(-\frac{\alpha}{8 \e^{2T}}  \frac{|\e^{-t^2}x - y|^2}{1 -
            \e^{-2t^2}} \biggr),
\end{align*}
using $t|x| \leq C$ in the last step.

For the bound on $\tilde{K}$ we consider
\begin{align*}
  t \partial_{x_i} \biggl[H_{\ell_i}(x_i) H_{m_i} \biggl(\frac{x_i \e^{-t} - y_i}{\sqrt{1
  - \e^{-2t}}} \biggr) \biggr] &=  t H_{m_i} \biggl(\frac{x_i \e^{-t} - y_i}{\sqrt{1
  - \e^{-2t}}} \biggr) \partial_{x_i} H_\ell(x_i)\\
  &\quad + H_{\ell_i}(x_i) t \partial_{x_i}  H_{m_i} \biggl(\frac{x_i \e^{-t} - y_i}{\sqrt{1
  - \e^{-2t}}} \biggr).
\end{align*}
So, as the first term on the right-hand side just decreases in degree
we look at
\begin{align*}
  t \partial_{x_i}  \biggl(\frac{x_i \e^{-t} - y_i}{\sqrt{1 - \e^{-2t}}}
  \biggr)^{m_i} =  m_i \biggl(\frac{x_i \e^{-t} - y_i}{\sqrt{1 - \e^{-2t}}}
  \biggr)^{m_i - 1} t \frac{\e^{-t}}{\sqrt{1 - \e^{-2t}}}
\end{align*}
The last term is bounded as $t \downarrow 0$, and the rest of the
proof is as before.
\end{proof}

\subsection*{Acknowledgments}
This work was partially supported by NWO-VICI grant 639.033.604 of the Netherlands Organisation for Scientific Research (NWO).

The author wishes to thank Alex Amenta and Mikko Kemppainen for inspiring discussions.

\bibliographystyle{plain}
\bibliography{library}

\end{document}